\documentclass[12pt]{amsart}
\usepackage{amscd,times,fullpage}

\newtheorem{thm}{Theorem}
\newtheorem{prop}{Proposition}
\newtheorem{lem}{Lemma}
\newtheorem{cor}{Corollary}

\theoremstyle{remark}
\newtheorem{rem}{Remark}

\theoremstyle{definition}

\newtheorem{ex}{Example}

\newcommand{\C}{\mathbb{ C}}

\newcommand{\OO}{\mathcal{ O}}

\newcommand{\PP}{\mathbb{ P}}
\newcommand{\R}{\mathbb{ R}}

\newcommand{\Z}{\mathbb{ Z}}

 \DeclareMathOperator{\dummygg}{\mathfrak{g}}
\renewcommand{\gg}{\dummygg}
\DeclareMathOperator{\hh}{\mathfrak{h}} \DeclareMathOperator{\TT}{\mathfrak{t}}

\title{Chern numbers and the geometry of partial flag manifolds}
\author{D.~Kotschick}
\address{Mathematisches Institut, Ludwig-Maximilians-Universit\"at
M\"unchen, Theresienstr.~39, 80333 M\"unchen, Germany}
\email{dieter@member.ams.org}
\author{S.~Terzi\'c}
\address{Faculty of Science, University of Montenegro, Cetinjski put bb, 81000 Podgorica, Montenegro}
\email{sterzic@cg.ac.yu}

\date{February 9, 2008; MSC 2000: primary 53C30, 57R20; secondary 14M15, 53C26, 53C55}

\begin{document}

\begin{abstract}
We calculate the Chern classes and Chern numbers for the natural almost Hermitian structures of the partial flag 
manifolds $F_n=SU(n+2)/S(U(n)\times U(1)\times U(1))$. For all $n>1$ there are two invariant complex algebraic 
structures, which arise from the projectivizations of the holomorphic tangent and cotangent bundles of $\C P^{n+1}$. 
The projectivization of the cotangent bundle is the twistor space of a Grassmannian considered as a quaternionic 
K\"ahler manifold. There is also an invariant nearly K\"ahler structure, because $F_n$ is a $3$-symmetric space. 
We explain the relations between the different structures and their Chern classes, and we prove that $F_n$ is not 
geometrically formal.
\end{abstract}

\maketitle

\section{Introduction}\label{s:intro}

This paper discusses the geometry of the homogeneous spaces
$$
F_n = SU(n+2)/S(U(n)\times U(1)\times U(1))
$$
from several points of view. These flag manifolds carry a number of interesting structures that we would like to understand. 
The relations between the different structures are quite intriguing. Note that $F_0$ is the $2$-sphere, and everything we will
say is either trivial or does not apply in this case. Next, $F_1$ is the manifold of complete flags in $\C^3$, of real dimension $6$.
This plays a special r\^ole in our discussion. The general case begins with $F_2$, of real dimension $10$. All $F_n$ with $n \geq 2$
are genuine partial flag manifolds.

We now briefly describe the different geometric features of $F_n$ that we shall consider.

\subsection{Complex structures}

It is a classical fact due to Borel, Koszul and Wang that $F_n$ admits at least one invariant K\"ahler structure. The 
starting point of this work was an observation of Borel and Hirzebruch~\cite{BH}, pointing out that $F_2$ has two different 
invariant structures as a complex projective variety, for which the values of the Chern number $c_1^5$ are different.
Extending this observation, we shall see that, up to conjugation and automorphisms, each $F_n$ with $n\geq 2$ has 
precisely two invariant complex structures. We shall give explicit formulae for their Chern classes and indicate how to 
calculate the Chern numbers in several different ways. As a particular application of these calculations we will see that 
the value of the Chern number $c_1^{2n+1}$ always distinguishes the two structures. 
For $n\leq 3$ we give the values of all the Chern numbers for the two complex structures.

One way of calculating the Chern numbers is through Lie theory, using the description of Chern classes as polynomials 
in the roots due to Borel and Hirzebruch~\cite{BH}. Another way, also used by Hirzebruch in his recent paper~\cite{H05}, 
is to look for a geometric interpretation of the complex structures on $F_n$, and to perform the calculations using differential 
or algebraic geometry. This works out very nicely because the two complex structures on $F_n$ are precisely those of the 
projectivizations of the holomorphic tangent and cotangent bundles of $\C P^{n+1}$. Moreover, the projectivization of the 
cotangent bundle carries a tautological complex contact structure,
and this identifies it with the total space of a certain $S^2$-bundle over the Grassmannian
$G_n = SU(n+2)/S(U(n)\times U(2))$, as first observed by Wolf~\cite{Wolf}.
With hindsight the Grassmannian is a quaternionic K\"ahler manifold in the sense of Salamon~\cite{S}, and the $S^2$-bundle 
over it is its twistor space. This relates our calculations of Chern numbers for the projectivization of the cotangent
bundle of $\C P^{n+1}$ to earlier calculations of the indices of certain elliptic operators on $G_n$, cf.~\cite{SW}.

Our initial motivation for the calculations of Chern numbers of the complex structures on $F_n$ was Hirzebruch's 
problem asking which linear combinations of Chern numbers are topological invariants of smooth projective varieties 
or of compact K\"ahler manifolds. This problem, originally raised in~\cite{Hir1}, was recently resolved completely in 
complex dimensions strictly smaller than $5$, see~\cite{Chern}, and we hope that the calculations performed in this
paper will be useful in studying this problem in higher dimensions. The calculations in complex dimension $5$, that 
is for $F_2$, summarized in the table in Figure~\ref{table2} might lead one to speculate about what happens for 
arbitrary $n$. In order to test such speculations we completed all the calculations for $n=3$, that is in complex dimension 
$7$. They are summarized in the table in Figure~\ref{table3} at the end of the paper. We also give closed formulae for
a few Chern numbers for arbitrary $n$ in Theorems~\ref{t:C} and~\ref{t:dualC}. 
Nevertheless, we do not pursue the applications to Hirzebruch's problem here.

\subsection{Generalized symmetric spaces and geometric formality}

According to Gray~\cite{G}, compare also~\cite{T0}, every $F_n$ endowed with the normal homogeneous metric induced
by the Killing form is a $3$-symmetric space. Generalizing the definition of symmetric spaces, this means that for every 
$p\in F_n$ there is a globally defined isometry $\theta\colon F_n\rightarrow F_n$ having $p$ as an isolated fixed point and 
satisfying $\theta^3 = Id$. 
More general $k$-symmetric spaces are defined in the same way by requiring $\theta$ to  be of order $k$.

A closed manifold is called geometrically formal if it admits a Riemannian metric for which all
wedge products of harmonic forms are harmonic; cf.~\cite{K}. Compact symmetric spaces provide examples of geometrically
formal manifolds because the harmonic forms for an invariant metric are precisely the invariant forms. This is no longer
true for $k$-symmetric spaces with $k>2$. In~\cite{KT} we showed that the structure of the cohomology ring of 
many $k$-symmetric spaces of the form $G/T$, where $T\subset G$ is a torus, is incompatible with geometric
formality. We will generalize the arguments from~\cite{KT}, which in particular showed that $F_1$ is not geometrically formal, 
to show that $F_n$ is not geometrically formal for all $n \geq 1$. Thus no Riemannian metric on $F_n$ has the property that
the harmonic forms are a subalgebra of the de Rham algebra. For invariant metrics this is not hard to see, and is of interest
in the context of Arakelov geometry, cf.~\cite{KuTa}.

\subsection{Nearly K\"ahler structures}\label{ss:nK}

The order $3$ symmetry $\theta$ of the normal homogeneous metric $g$ on $F_n$ can be identified with an automorphism
of $G=SU(n+2)$ fixing the subgroup $H=S(U(n)\times U(1)\times U(1))$. The derivative 
of $\theta$, also denoted $\theta$, acts as an automorphism of the Lie algebra $\gg$ with fixed point set $\hh$.
Although $\theta - Id$ is not invertible on $\gg$, it is invertible on $T_p F_n = \gg/\hh$. Therefore,
$$
0 = \theta^3 - Id = (\theta - Id)(\theta^2 + \theta + Id)
$$
implies $\theta^2 + \theta + Id=0$ on $T_p F_n = \gg/\hh$. Now, on $T_p F_n = \gg/\hh$, one can define
$$
J_{\theta} = \frac{1}{\sqrt 3}(Id + 2\theta) \ .
$$
This is an isometry of $g$ satisfying $J_{\theta}^2 = -Id$, as follows immediately from $\theta^2 + \theta + Id=0$.
Thus $J_{\theta}$ is an almost complex structure and $(g,J_{\theta})$ is an almost Hermitian structure called the 
canonical almost Hermitian structure of the $3$-symmetric space.

Gray~\cite{G} proved that the canonical almost complex structure of a $3$-symmetric space 
is nearly K\"ahler, i.~e.
$$
(\nabla_v J_{\theta})v = 0
$$
for all vector fields $v$, where $\nabla$ denotes the Levi-Civit\`a connection of $g$. 
Conversely, Butruille~\cite{Bu} recently proved that every homogeneous nearly K\"ahler structure that is 
not K\"ahler comes from a $3$-symmetric space.

The nearly K\"ahler structure of a $3$-symmetric space is K\"ahler, meaning
$\nabla_v J_{\theta}=0$ for all $v$, if and only if it is Hermitian symmetric. 
As $F_n$ is not a symmetric space for any $n\geq 1$, its nearly K\"ahler structure can not be K\"ahler. We will
see that the $J_{\theta}$ defined above is the unique (up to conjugation) non-integrable invariant almost complex 
structure on $F_n$. Moreover, this structure is a special case of non-integrable almost complex structures on
twistor spaces introduced by Eells and Salamon~\cite{ES}; compare also~\cite{Bianchi,AGI,Nagy1}.
We shall compute its Chern classes and compare the Chern numbers to those of the integrable 
K\"ahlerian structures. See Figures~\ref{table2} and~\ref{table3} for a summary of these calculations for 
$n=2$ and $3$.


\begin{figure}
{
\begin{tabular}{|l|c|c|c|} \hline
& & & \\ 
& \raisebox{1.5ex}[-1.5ex]{standard structure $\PP (T\C P^3)$} & \raisebox{1.5ex}[-1.5ex]{twistor space $\PP (T^{*}\C P^3)$} & \raisebox{1.5ex}[-1.5ex]{nearly K\"ahler structure} \\ \hline\hline
{\rule [-3mm]{0mm}{8mm}$c_1^5$ } 
& $4500$ & $4860$ & $-20$  \\ \hline
{\rule [-3mm]{0mm}{8mm}$c_1^3c_2$ } 
& $2148$ & $2268$ & $-4$  \\ \hline
{\rule [-3mm]{0mm}{8mm}$c_1c_2^2$ } 
& $1028$ & $1068$ & $-4$ \\ \hline
{\rule [-3mm]{0mm}{8mm}$c_1^2c_3$ } 
& $612$ & $612$ & $20$ \\ \hline
{\rule [-3mm]{0mm}{8mm}$c_2c_3$ } 
& $292$ & $292$ & $4$ \\ \hline
{\rule [-3mm]{0mm}{8mm}$c_1c_4$ } 
& $108$ &  $108$ & $12$ \\ \hline
{\rule [-3mm]{0mm}{8mm}$c_5$ } 
& $12$ & $12$ & $12$ \\ \hline
\end{tabular}
}
\caption{The Chern numbers of the invariant almost Hermitian structures on $F_2$}\label{table2}
\end{figure}


\subsection{Einstein metrics}

It is well known that the six-dimensional flag manifold $F_1$ has exactly two invariant Einstein metrics, up to scale
and isometry; see for example~\cite{Ar,BK,Kimura}. One of these is K\"ahler--Einstein, compatible with the essentially
unique integrable complex structure, and the other one is non-K\"ahler, but almost Hermitian. This is the normal metric, 
which, as explained above, is  nearly K\"ahler because it is $3$-symmetric. 

For $n\geq 2$, there are precisely three invariant Einstein metrics on
$F_n$, up to scale and isometry. This is due to Arvanitoyeorgos~\cite{Ar} and Kimura~\cite{Kimura}. Two of the three
Einstein metrics are K\"ahler--Einstein, compatible with the two different invariant complex structures. The third Einstein
metric, which is not K\"ahler, is not the normal nearly K\"ahler metric. In~\cite{G}, Gray had claimed that the normal metric 
of any $3$-symmetric space is Einstein, but, in~\cite{G2}, he himself corrected this, and mentioned that the normal 
metric of $F_2$ is not Einstein. It is a result of Wang and Ziller~\cite{zillerENS} that the normal metric on $F_n$ is 
not Einstein for all $n\geq 2$.

Both the invariant nearly K\"ahler metric and the invariant Einstein non-K\"ahler metric on $F_n$ are 
obtained from the K\"ahler--Einstein metric of the twistor space by scaling the $S^2$ fibers, but the scaling factor
is different for the two metrics; see~\cite{AGI}.

\subsection*{Outline of the paper}

In Section~\ref{s:Lie} we recall some facts from the theory developed by Borel and Hirzebruch in~\cite{BH}
and apply them to determine the invariant almost complex structures of $F_n$ and discuss their integrability.
In Section~\ref{s:ChernLie} we calculate Chern classes and Chern numbers for these structures using Lie
theory.

In Section~\ref{s:cx} we discuss the complex and the nearly K\"ahler structures of $F_n$ without using Lie theory.
Our point of view here is complementary to that of Section~\ref{s:Lie}, and relies on the work of Salamon and his coauthors~\cite{S,Bianchi,ES,LeS}; compare also~\cite{AGI,Besse,Nagy1}.
This section, and Sections~\ref{s:coho} to~\ref{ss:Chernnearly} which are based on it, can be read
independently of Sections~\ref{s:Lie} and~\ref{s:ChernLie}, except for a few isolated remarks aimed at relating the two points of
view. In Section~\ref{s:coho} we give a simple description of the cohomology ring of $F_n$ and use it to prove a 
general result about Hodge and Chern numbers for arbitrary K\"ahlerian complex
structures on manifolds with this cohomology ring, and we also prove that $F_n$ is not geometrically formal.
Sections~\ref{s:standard} to~\ref{ss:Chernnearly} contain calculations of Chern numbers for the three different
almost Hermitian structures.

In Section~\ref{s:final} we comment on the relations between the different points of view.

\section{The Lie theory of generalized flag manifolds}\label{s:Lie}

The partial flag manifolds $F_n$ are a special subclass of the so-called generalized flag manifolds, which are 
homogeneous spaces of the form $G/H$, with $G$ a compact connected semisimple Lie group and $H\subset G$
a closed subgroup of equal rank that is the centralizer of a torus. For such generalized flag manifolds 
the cohomology ring, the invariant almost complex structures and their Chern classes, their integrability, 
and the invariant Einstein metrics can be described explicitly in the framework of the theory initiated by Borel
and Hirzebruch~\cite{BH}; see also~\cite{Besse,T,zillerENS,W-G}. We recall some aspects of this theory
relevant to our calculations of Chern numbers on $F_n$. For a different point of view on some of these matters, the 
reader may consult~\cite{MN}.

\subsection{Some general theory}

For a compact homogeneous space $G/H$ as above one has the isotropy representation of $H$ on
$T_{eH}(G/H)$, which can be decomposed into a direct sum of irreducible summands. By Schur's lemma 
a $G$-invariant metric on $G/H$ restricts to each of the irreducible summands as a constant multiple of 
the Killing form. Conversely, any choice of positive-definite multiples of the Killing form for each irreducible 
summand uniquely specifies a $G$-invariant metric on $G/H$. The determination of invariant Einstein
metrics in~\cite{zillerENS,Kimura,Ar} proceeds by solving--if possible--the algebraic system for the multiples 
of the Killing form given by the equation making the Ricci tensor proportional to the metric. For example, if 
the isotropy representation is irreducible, then the normal homogeneous metric given by the Killing form is 
Einstein, and is the only invariant Einstein metric. For the partial flag manifolds $F_n$ the isotropy representation 
decomposes into the direct sum of three irreducible summands, one of real dimension $2$ and two of real 
dimension $2n$, see~\eqref{eq:isotropy} below. For $F_1$ the method of Wang and Ziller~\cite{zillerENS} 
yields exactly two non-isometric non-homothetic invariant Einstein metrics, and for $F_n$ with $n\geq 2$ it 
yields three; compare~\cite{Ar,Kimura}.

The invariant almost complex structures on generalized flag manifolds can be enumerated in the same way, 
as they correspond to complex structures on the vector space $T_{eH}(G/H)$ invariant under the isotropy 
representation. Note that every generalized flag manifold admits an invariant complex structure, as its isotropy 
subgroup is the centralizer of a torus. Therefore, again by Schur's lemma, any isotropy invariant complex 
structure in $T_{eH}(G/H)$ is unique up to conjugation on every irreducible summand of the isotropy
representation. Thus, if the isotropy representation decomposes into $p$ irreducible summands, each admitting 
a complex structure, then the number of invariant almost complex structures on $G/H$ is $2^p$. If we identify
complex conjugate structures this leaves $2^{p-1}$ invariant almost complex structures, but some of these may
still be equivalent under automorphisms of $G$. For the partial flag manifolds $F_n$ we have $p=3$, so up 
to conjugation there are always $2^{3-1}=4$ invariant almost complex structures. However, it will turn out that 
after taking into account automorphisms we are left with only $3$ almost complex structures for $n\geq 2$.
For $n=1$ two of the three are equivalent under an additional automorphism that is not present in the general case.

This enumeration of invariant almost complex structures is too crude to determine which ones are integrable, and 
for the calculation of Chern classes. Following Borel and Hirzebruch~\cite{BH} one deals with these two points 
using the roots of the Lie algebra $\gg$ with respect to a Cartan subalgebra.

Let $\TT\subset\gg$ be a Cartan subalgebra for $\gg$ and $\hh$ . This gives rise to
a root space decomposition
$$
\gg = \hh\oplus\gg_{\pm\beta_1}\oplus\ldots\oplus\gg_{\pm\beta_k} \ ,
$$
where the $\pm\beta_i$ are the complementary roots for $\hh\subset\gg$. Note that $T_{eH}(G/H)$ is 
identified with $\gg_{\pm\beta_1}\oplus\ldots\oplus\gg_{\pm\beta_k}$. 

Now any isotropy invariant complex structure on $T_{eH}(G/H)$ is also invariant under the adjoint representation
of a maximal torus, and therefore induces a complex structure on each root space $\gg_{\pm\beta_j}$. Comparing
this orientation on $\gg_{\pm\beta_j}$ with the orientation given by the adjoint representation, one assigns a sign 
$\pm 1$ to $\gg_{\pm\beta_j}$. Note, further, that each irreducible summand of the isotropy representation is a sum 
of some of these root spaces. Therefore, invariant almost complex structures on $G/H$ are specified by choices of 
signs for the complementary roots compatible with the irreducible summands of the isotropy representation.

The following three lemmata are due to Borel and Hirzebruch~\cite{BH}.
\begin{lem}\label{l:integrable}{\rm (\cite{BH}, 13.7)}
An invariant almost complex structure is integrable if and only if one can find an ordering on the coordinates for the  
Cartan algebra such that its corresponding system of complementary roots is positive and closed in the sense that 
whenever $\alpha$ and $\beta$ are complementary roots and $\alpha+\beta$ is a root, then $\alpha+\beta$ is a 
complementary root.
\end{lem}

\begin{lem}\label{l:Chern}{\rm (\cite{BH}, 10.8)}
For an invariant almost complex structure $J$ its complementary roots $\beta_i$ considered as elements 
of $H^2(G/H)$ are the Chern roots, i.e.~the total Chern class is 
$$
c(T(G/H),J)=\prod_{i=1}^{k}(1+\beta_i) \ .
$$
\end{lem}

\begin{lem}\label{l:Hodge}{\rm (\cite{BH}, 14.10)}
Every invariant integrable almost complex structure makes $G/H$ into a rational projective algebraic manifold 
over $\C$, all of whose cohomology is of Hodge type $(p,p)$.
\end{lem}

\begin{rem}\label{equivalence}
It is also proved in~\cite{BH}, 13.7, that if for two invariant complex structures on $G/H$ there is an automorphism of 
the Cartan algebra $\TT$ which carries the root system of one structure into that of the other structure and fixes the 
root system of $H$, then these two structures are equivalent under an automorphism of $G$ fixing $H$. 
\end{rem}

\subsection{Application to the partial flag manifolds $F_n$}\label{ABH}

We now specialize this general discussion to the consideration of $F_n$ with $G=SU(n+2)$ and 
$H=S(U(1)\times U(1)\times U(n))$. At the level of Lie algebras this means that we consider 
$A_{n+1}/(\TT^2\oplus A_{n-1})$, with the specific embedding of  the subalgebra given by the 
$3$-symmetric structure, see for example~\cite{W-G,T}.

\begin{prop}
The cohomology ring of $F_n$ is 
\begin{equation}\label{cohom}
H^{*}(F_n)= \left(\R [x,y,\sum _{i=1}^{n}y_{i}^{2},\ldots ,\sum _{i=1}^{n}y_{i}^{n}]\right)/ \langle P_2,\ldots ,P_{n+2} \rangle \ ,
\end{equation}
where 
\begin{equation}\label{eq:rel}
P_k = (-y)^{k} + (-x + y)^{k} + \sum _{i=1}^{n}(y_i + \frac{1}{n}x)^{k},  \;\; \textrm{for} \  \ k = 2,\ldots ,n+2 \ .
\end{equation}
\end{prop}
\begin{proof}
Let $x_1,\ldots,x_{n+2}$ be canonical coordinates on the maximal Abelian subalgebra of $A_{n+1}$ with 
$x_1+\ldots+x_{n+2}=0$. Analogously, let $x,y$ be linear coordinates on $\TT ^2$ and $y_1,\ldots y_{n}$ 
canonical coordinates on $A_{n-1}$ with $y_1+\ldots+y_{n}=0$.
It is not hard to see that the relations between $x_1,\ldots ,x_{n+2}$ and $x,y,y_1,\ldots ,y_{n}$ are as follows
(cf.~\cite{T}):
\begin{equation}\label{restr}
\begin{split}
x_{i} &= y_{i} + \frac{1}{n}x\quad \textrm{for} \quad 1\leq i\leq n \ ,\\
x_{n+1} &= -x + y, \quad x_{n+2} = -y \ .
\end{split}
\end{equation}
Cartan's theorem on the cohomology of compact homogeneous spaces together with the relations~\eqref{restr} 
implies that for $G=SU(n+2)$ and $H=S(U(1)\times U(1)\times U(n))$ the cohomology ring of the quotient is given 
by~\eqref{cohom} and~\eqref{eq:rel}.
\end{proof} 
The relations $P_k$ for $k = 2, \ldots , n$ eliminate the cohomology generators
$\sum_{i=1}^{n}y_i^2,\ldots,\sum_{i=1}^{n}y_i^{n}$. It follows that $H^{*}(F_n)$ is generated by the two generators 
$x$ and $y$ of degree $2$, with relations in degrees $n+1$ and $n+2$.

\begin{lem}\label{l:compl}
The complementary roots for $A_{n+1}$ with respect to $\TT^{2}\oplus A_{n-1}$ are, up to sign, the following:
\begin{equation}\label{eq:compl}
\begin{split}
x_{n+1}-x_{n+2} &= 2y-x \ , \ \textrm{and} \\ 
x_i-x_{n+1} &=y_{i} +\frac{n+1}{n}x - y \ \  \ \textrm{for} \quad 1\leq i\leq n \ , \ \\
x_i-x_{n+2} &=y_i + \frac{1}{n}x + y \ \  \ \textrm{for} \quad 1\leq i\leq n \ .
\end{split}
\end{equation}
\end{lem}
\begin{proof}
The roots for the algebra $A_{n+1}$ are $\pm (x_i - x_j)$, $1\leq i< j\leq n+2$, and for the subalgebra $A_{n-1}$ the 
roots are $\pm (y_i -y_j)$, $1\leq i< j\leq n$ . Using the relations~\eqref{restr}, we can express the roots for $A_{n+1}$
in the form
$$
\pm (y_i - y_j), \; \pm (y_i + \frac{1}{n}x + y), \; \pm (x - 2y), \; \pm (y_{i} + \frac{n+1}{n}x - y) \ .
$$
From this it is clear that the complementary roots are given by~\eqref{eq:compl}.
\end{proof}
We now deduce the following classification of invariant almost complex structures.
\begin{prop}\label{roots}
The homogeneous space $F_n$ admits at most three invariant almost complex structures $I$, $J$ and $\hat J$, 
up to equivalence and conjugation. Their roots are:
\begin{equation}\label{eq:roots}
\begin{split}
I \ &\colon \ \quad  y_i+\frac{n+1}{n}x -y,\;  y_{i}+\frac{1}{n}x+y,\;  -x+2y,\;  1\leq i\leq n \ , \\
J \ &\colon \ \quad  y_i+\frac{n+1}{n}x-y,\; -y_i-\frac{1}{n}x-y,\; x-2y,\;  1\leq i\leq n \ , \\
{\hat J} \ &\colon \ \quad  y_i+\frac{n+1}{n}x-y,\; -y_i-\frac{1}{n}x-y,\; -x+2y,\; 1\leq i\leq n  \ .
\end{split}
\end{equation}
The structures $I$ and $J$ are integrable, and $\hat J$ is not.
\end{prop}
We will see that for $n=1$ the integrable structure $I$ is equivalent to the complex conjugate of $J$ under an automorphism, 
whereas for $n\geq 2$ this is no longer true; in fact the different integrable structures are then distinguished by their Chern 
classes.
The non-integrable invariant almost complex structure is the natural nearly K\"ahler structure arising from the $3$-symmetric
structure, as discussed in Subsection~\ref{ss:nK} above.
\begin{proof}
From the description of the complementary roots in Lemma~\ref{l:compl} it follows that the isotropy representation 
decomposes into the direct sum of the following three irreducible summands:
\begin{equation}\label{eq:isotropy}
R_0=\gg _{x_{n+1}-x_{n+2}}, \quad R_1=\gg _{x_1-x_{n+1}}\oplus \ldots \oplus \gg _{x_{n} -x_{n+1}}, 
\quad R_2=\gg _{x_1-x_{n+2}}\oplus \ldots \oplus \gg _{x_{n}-x_{n+2}} \ .
\end{equation}    
It follows that, up to conjugation, there are $4$ invariant almost complex structures. Their roots are given by choosing 
signs for the irreducible summands of the isotropy representation. Up to conjugation, we may take the following signs:
\begin{enumerate}
\item[(a)] $R_0^{+}, R_1^{+}, R_2^{+}$,
\item[(b)] $R_0^{-}, R_1^{+}, R_2^{+}$,
\item[(c)] $R_0^{-}, R_1^{+}, R_2^{-}$,
\item[(d)] $R_0^{+}, R_1^{+}, R_2^{-}$.
\end{enumerate}
The invariant almost complex structures given by the first three choices are integrable by Lemma~\ref{l:integrable}, 
as their roots correspond to the orderings $x_1<\ldots<x_{n+1}<x_{n+2}$, $x_1<\ldots <x_{n+2}<x_{n+1}$ and 
$x_{n+2}<x_1<\ldots <x_{n+1}$, respectively. Moreover, the automorphism of the maximal Abelian subalgebra for $A_{n+1}$ 
given by interchanging $x_{n+1}$ and $x_{n+2}$ maps the complementary roots of the first structure to the complementary 
roots of the second structure and leaves the root system of $A_{n-1}$ invariant. Therefore, it follows from Remark~\ref{equivalence} 
that these two structures are equivalent under an automorphism of the homogeneous space.

The fourth structure is not integrable, as there is no ordering on the coordinates $x_1,\ldots ,x_{n+2}$ for which the roots defining 
this structure are positive. 

Thus there are two integrable and one non-integrable invariant almost complex structure on $F_n$. The roots in~\eqref{eq:roots}
arise from (a) for $I$, (c) for $J$ and (d) for $\hat J$ by combining~\eqref{eq:isotropy} (with the appropriate signs) with~\eqref{eq:compl}.
\end{proof}

\section{Chern numbers from Lie theory}\label{s:ChernLie}

We now calculate certain Chern classes and Chern numbers of the invariant almost complex structures on $F_n$ using
Lie theory.

\begin{ex}\label{ex:c1}
The first Chern classes for the  structures $I$, $J$ and $\hat J$ are obtained immediately from Lemma~\ref{l:Chern} using 
the description of the corresponding roots in Proposition~\ref{roots}. The result is:
\[
c_1(I)=(n+1)x+2y,\quad c_1(J)=(n+1)(x-2y),\quad c_1(\hat J)=(n-1)(x-2y) \ .
\]
\end{ex}

\begin{ex}\label{ex:c2}
The cohomology relation $P_2$ from~\eqref{eq:rel} gives that
\[
\sum_{i=1}^{n}y_i^2 = 2xy-2y^2-\frac{n+1}{n}x^2 \ .
\]
Using this, we compute the second Chern classes for the invariant almost structures using Lemma~\ref{l:Chern} and the 
description of their roots given by Proposition~\ref{roots}. The result is:
\begin{alignat*}{1}
c_2(I) &=\frac{n(n+1)}{2}x^2+(3n+2)xy+(2-n)y^2 \ , \\
c_2(J) &=\frac{n(n+1)}{2}x^2+(-2n^2-3n-2)xy+(2n^2+3n+2)y^2 \ , \\
c_2({\hat J}) &=\frac{n(n-3)}{2}x^2+(-2n^2+5n-2)xy+(2n^2-5n+2)y^2 \ .
\end{alignat*}
\end{ex}  

\begin{rem}
Proposition~\ref{roots} implies that for even $n$ the structures $I$ and $\hat J$ define the same orientation on $F_n$, while $J$ 
defines the opposite orientation. For odd $n$, the orientations given by $I$ and $J$ are the same, while the one given by $\hat J$ is 
different. This fact will show up in the sign of their top Chern classes in the calculations we provide below for $n=2$ and $n=3$. 
\end{rem} 

To calculate Chern numbers explicitly we now consider the cases where $n$ is small.

\subsection{The complete flag manifold $F_1$}\label{ss:complete}

The case $n=1$ is special because the three irreducible summands of the isotropy representation are 
all of the same real dimension equal to two. The map interchanging $x_1$ and $x_2$ but fixing $x_3$
defines an automorphism of $A_2$ fixing the Abelian subalgebra $\TT^2$ which interchanges the invariant 
complex structure $I$ and the complex conjugate of $J$ on $F_1$. It follows in particular that 
they have the same Chern numbers.

In this case the two cohomology generators $x$ and $y$ satisfy the relations $y^2-xy+x^2=0$ and 
$-y^3+(y-x)^3+x^3=0$ obtained from~\eqref{eq:rel} by setting $k=1$ and $k=2$ respectively. The second 
relation simplifies to $xy^2=x^2y$, which together with the first relation implies $x^3=y^3=0$.
Using the relations we find, in addition to $c_1(I)=2(x+y)$ from Example~\ref{ex:c1}, that $c_2(I)=6xy$
and  $c_3(I)=6x^2y$. Clearly the topological Euler characteristic is $6$, so $x^2y=xy^2$ is the positive 
cohomology generator in top degree with respect to the orientation defined by the complex structure $I$.
Multiplying out and using the relations again we find the well known values for the Chern numbers:
$c_1c_2(I)=24$ and $c_1^3(I)=48$.

For the non-integrable invariant almost complex structure $\hat J$ we already know $c_1(\hat J)=0$ 
by setting $n=1$ in the formula in Example~\ref{ex:c1}. Thus $c_1c_2(\hat J)=c_1^3(\hat J)=0$. 

By the discussion in Subsection~\ref{ss:nK}, the non-integrable invariant almost complex structure
$\hat J$ is nearly K\"ahler because $A_{2}$ and $\TT^2$ form a $3$-symmetric pair. It is a result of 
Gray~\cite{G2} that every non-K\"ahler nearly K\"ahler manifold of real dimension $6$ has vanishing 
first Chern class.

\subsection{The case $n=2$} 

This is the example mentioned first by Borel and Hirzebruch in~\cite{BH}, 13.9 and 24.11, and then in~\cite{H05}.
There only the values of $c_1^5$ are given for two different invariant complex structures. These are the $I$ 
and $J$ discussed above, and we now give complete calculations for their Chern numbers.

The formula~\eqref{eq:rel} for the cohomology relations gives, for $k=2$, that $y_1=z$ satisfies 
$z^{2}=-\frac{1}{2}(\frac{3}{2}x^{2}-2xy+2y^{2})$.
Therefore, using again~\eqref{cohom} and~\eqref{eq:rel}, we see that the generators $x$ and $y$ of the 
cohomology algebra of $F_2$ satisfy  the relations
$$
x^{3}=2(x^{2}y-xy^{2}), \quad y^{4}=0 \ .
$$
This gives the following relations in top degree cohomology:
$$
y^5=xy^4=0,\quad x^4y=x^3y^2=2x^2y^3 \ .
$$

Using Lemma~\ref{l:Chern} and the relations in
cohomology we find that the Chern classes of $I$ are as follows:
$$
c_1(I)=3x+2y,\quad c_2(I)=3x^2+8xy,\quad c_3(I)=-x^3+14x^2y \ ,
$$
$$
c_4(I)=-14x^3y+14x^2y^2-8xy^3,\quad c_5(I)=12x^2y^3 \ .
$$
Since the Euler characteristic of this space is $12$, we obtain the Chern numbers
given in the first column of the table in the introduction.

For $J$ we find in the same way
$$
c_1(J)=3(x-2y),\quad c_2(J)=3x^2-16xy+16y^2,\quad c_3(J)=x^3-14x^2y+36xy^2-24y^3 \ ,
$$
$$
c_4(J)=-2x^3y+22x^2y^2-40xy^3,\quad c_5(J)=-12x^2y^3 \ .
$$
Multiplying out and using the relations in cohomology this leads to the Chern numbers given in the 
second column of the table in Figure~\ref{table2}.

There are several ways to check that we have not made numerical mistakes in the calculations. First of all, as 
explained in Sections~\ref{s:standard} and~\ref{ss:Cherntwistor} below, all these numbers can be calculated in 
a completely different 
way without using Lie theory, and that calculation leads to the same results. Second of all, the Chern numbers 
must satisfy certain relations imposed by the Hirzebruch--Riemann--Roch theorem. For the arithmetic genus 
of our five-fold $F_2$, HRR gives
    $$
    \sum_{q=0}^{5}(-1)^qh^{0,q}=
    \frac{1}{1440}(-c_{1}c_{4}+c_{1}^{2}c_{3}+3c_{1}c_{2}^{2}-c_{1}^{3}c_{2}) \ .
    $$
The left hand side is $=1$ because all the cohomology is of type $(p,p)$ by Lemma~\ref{l:Hodge}. Substituting 
the values of the Chern numbers computed above into the right hand side provides a non-trivial consistency check.

In the same way as we did the calculation for the integrable structures, we can also compute the Chern classes 
for the non-integrable invariant complex structure $\hat J$:
$$
c_1(\hat J) = x-2y, \quad c_2(\hat J) = -x^2, \quad c_3(\hat J) = -x^3 +6x^2y - 12 xy^2 +8y^3 \ , 
$$
$$
c_4(\hat J) = 6x^3y-18x^2y^2 + 24xy^3, \quad c_5(\hat J) = 12x^2y^3 \ .   
$$
This gives the Chern numbers in the third column of the table in Figure~\ref{table2}.

\subsection{The case $n=3$}\label{ss:3}

Now we consider $F_3$, of real dimension $14$. From~\eqref{cohom} it follows that its real 
cohomology algebra has two generators $x$ and $y$ of degree $2$ and the relations in degree $4$ and $6$ are 
$$
\sum _{i=1}^{3}y_i^2 = 2xy -\frac{4}{3}x^2-2y^2 \ ,\ \sum
_{i=1}^{3}y_i^3 = \frac{20}{9}x^3-5x^2y+5xy^2 \ .
$$
Taking into account these expressions, the relations in degree $8$ and $10$ produce the following 
relations between $x$ and $y$:
$$
x^4+y^4-3x^3y+4x^2y^2-2xy^3, \;\; y^5=0 \ .
$$
Therefore in degree $12$ we get that $x$ and $y$ satisfy the following relations:
$$
y^6=xy^5=0, \quad x^4x^2=-4x^2y^4+3x^3y^3 \ ,
$$
$$
x^5y=-10x^2y^4+5x^3y^3,\quad x^6=-15x^2y^4+5x^3y^3 \ .
$$
This implies that in top degree cohomology we have 
$$
y^7=xy^6=x^2y^5=0,\quad x^4y^3=3x^3y^4, \quad x^5y^2=5x^3y^4 \ ,
$$
$$
 x^6y=5x^3y^4,\quad  x^7=0 \ .
$$
 
From Lemma~\ref{l:Chern} and the cohomology relations we find for $I$:
$$
c_1(I)=4x+2y, \quad c_2(I) =6x^2+11xy-y^2, \quad c_3(I) = 4x^3+21x^2y+3xy^2-2y^3 \ ,
$$
$$
c_4(I) = -5x^4+35x^3y-5x^2y^2, \quad c_5(I) = 5x^4y+25x^3y^2-10x^2y^3 \ ,
$$
$$
c_6(I) = 15x^4y^2-5x^3y^3, \quad c_7(I) = 20x^3y^4 \ .
$$
Now by a direct calculation one can obtain all the Chern numbers for $I$ in this case.
These are the numbers contained in the first column of the table in Figure~\ref{table3}.

We can also calculate the Chern classes and Chern numbers for $J$ and for $\hat J$ in the same 
way. We obtain that the Chern classes for $J$ are:
$$
c_1(J) =4x-2y, \quad c_2(J) =6x^2-29xy+29y^2, \quad c_3(J) = 4x^3-39x^2y+93xy^2-62y^3 \ ,
$$
$$
c_4(J) = -85x^4+235x^3y-235x^2y^2, \quad c_5(J) = 1095x^4y-1245x^3y^2+230x^2y^3 \ ,  
$$
$$
c_6(J) =-30x^4y^2+50x^3y^3, \quad c_7(J) =20x^3y^4 \ .
$$
For $\hat J$ the Chern classes are given by:
$$
c_1(\hat J) =2x-4y, \quad c_2(\hat J) =-5xy+5y^2, \quad c_3(\hat J) = -2x^3+7x^2y-9xy^2+6y^3 \ , 
$$
$$
c_4(\hat J) = 25x^4-65x^3y+65x^2y^2, \quad c_5(\hat J) = -45x^4y+115x^3y^2-110x^2y^3 \ , 
$$
$$
c_6(\hat J) =15x^4y^2-25x^3y^3, \quad c_7(\hat J) = -20x^3y^4 \ . 
$$
The Chern classes for $J$ and $\hat J$ lead to the second and third columns in the table in Figure~\ref{table3}.

\section{The complex geometry of $F_n$}\label{s:cx}

We now give geometric descriptions of the almost Hermitian structures of $F_n$ without using Lie theory or Gray's results 
on the structure of $3$-symmetric spaces.

We think of $F_n$ as being a partial flag manifold, as follows:
$$
F_n = \{ (L,P) \ \vert \ P \ \textrm{a} \ 2-\textrm{plane in} \  \C^{n+2}, \ L \ \textrm{a line in} \ P\} \ .
$$
This has a natural complex projective-algebraic structure with ample anti-canonical bundle and a
K\"ahler--Einstein metric of positive scalar curvature. With respect to this complex structure, there are two
forgetful holomorphic maps, mapping a pair $(L,P)$ to either $L$ or $P$. On the one hand, the map to $L$ gives a
fibration
$$
p\colon F_n \longrightarrow \C P^{n+1} \ ,
$$
exhibiting $F_n$ as the projectivized tangent bundle of $\C P^{n+1}$. On the other hand, the map $(L,P)\mapsto P$
defines a holomorphic fibration
$$
\pi\colon F_n\longrightarrow G_n \ ,
$$
where $G_n = Gr_{2,n}$ is the Grassmannian of complex $2$-planes in $\C^{n+2}$. The fiber of $\pi$ is $\C P^1$.

The Grassmannian
$$
G_n =  SU(n+2)/S(U(n)\times U(2))
$$
has a homogeneous complex structure, which is unique up to conjugation. With respect to this structure $G_n$ is a
Hermitian symmetric space and carries a K\"ahler--Einstein metric. This metric is quaternionic K\"ahler in the sense of 
Salamon~\cite{S}, meaning that its reduced holonomy group is contained
in $Sp(n)\cdot Sp(1)$, the quotient of $Sp(n)\times Sp(1)$ by the subgroup $\{\pm (1,1)\}\cong\Z_2$. Moreover,
the scalar curvature of this metric is positive.

To any quaternionic K\"ahler $4n$-manifold $M$ with positive scalar curvature, Salamon~\cite{S} associates a twistor space 
$Z$, which is the total space of a certain $S^2$-bundle over $M$, together with a complex structure $J$ and compatible 
K\"ahler--Einstein metric $g$ of positive scalar curvature on the total space, with the following properties, see~\cite{S,LeS}:
\begin{itemize}
\item the projection $\pi\colon Z\longrightarrow M$ is a Riemannian submersion with totally geodesic fibers of constant
Gaussian curvature,
\item the fibers of $\pi$ are holomorphic curves in $Z$ (although $M$ is not usually complex, so that $\pi$ is not holomorphic
in any sense), and
\item the orthogonal complement $D$ of the tangent bundle along the fibers $T\pi$ is a holomorphic contact 
distribution.
\end{itemize}
Just like in the real case, a contact distribution is a maximally non-integrable hyperplane distribution. Given $D\subset TZ$,
let $L=TZ/D$ be the quotient line bundle and $\alpha\colon TZ\longrightarrow L$ the  projection with kernel $D$. The
maximal non-integrability of $D$ means that if we think of $\alpha$ as a one-form with values in $L$, then the 
$(2n+1)$-form $\alpha\wedge (d\alpha)^n$ with values in $L^{n+1}$ is no-where zero. Thus $\alpha\wedge (d\alpha)^n$
is an isomorphism between the anti-canonical bundle $K^{-1}$ and $L^{n+1}$. In particular $c_1(Z)=c_1(K^{-1})=(n+1)c_1(L)$.
This relation, for arbitrary holomorphic contact manifolds, was already observed by Kobayashi~\cite{Kobayashi}.

As the fibers of $\pi$ are holomorphic curves in $Z$, the tangent bundle along the fibers $T\pi$ is a complex
line bundle over $Z$. The projection $\alpha$ gives an isomorphism between $T\pi$ and $L$. Thus, disregarding the 
holomorphic structure, we have an isomorphism $TZ\cong L\oplus D$ of complex vector bundles over $Z$.

Following Eells and Salamon~\cite{ES} one can define another almost Hermitian structure $(\hat J,\hat g)$ on $Z$ as follows. 
With respect to $\hat g$ the subbundles $T\pi$ and $D$ of $TZ$ are orthogonal, and $\hat g$ agrees with $g$ on $D$. For 
$v,w\in T\pi$ we define $\hat g(v,w)=\frac{1}{2}g(v,w)$. The subbundles $T\pi$ and $D$ of $TZ$ are invariant under $\hat J$, 
and $\hat J$ agrees with $J$ on $D$ and agrees with $-J$ on $T\pi$. According to~\cite{AGI,Nagy1},
the pair $(\hat J,\hat g)$ satisfies $(\nabla_v \hat J)v = 0$ for all vector fields $v$, where $\nabla$ denotes the Levi-Civit\`a 
connection of $\hat g$.
This precisely means that $(\hat J,\hat g)$ is a nearly K\"ahler structure in the sense of Gray~\cite{G2}. Note that by the 
definition of $\hat J$, the complex vector bundle $(TZ,\hat J)$ is isomorphic to $L^{-1}\oplus D$. This will allow us to 
determine the Chern classes of the nearly K\"ahler structure from those of the twistor space structure.

\section{The cohomology ring and some consequences}\label{s:coho}

We can easily describe the cohomology ring of $F_n$ explicitly using its description as the projectivization of the tangent
bundle of $\C P^{n+1}$.
\begin{prop}\label{p:coho}
The cohomology ring of $F_n$ is generated by two elements $x$ and $y$ of degree $2$, subject to the relations
\begin{equation}\label{eq:relFn}
x^{n+2}=0 \ , \qquad \ \frac{(x+y)^{n+2}-x^{n+2}}{y} = 0 \ .
\end{equation}
\end{prop}
\begin{proof}
Consider the fibration $p\colon F_n\rightarrow \C P^{n+1}$ given by the projectivization of the tangent bundle of
$\C P^{n+1}$. Let $x=p^*(H)$ denote the pullback of the hyperplane class. Then $x^{n+2}=0$ for dimension reasons.
Let $y$ be the tautological class on the total space, restricting to the hyperplane class on every fiber. By the
Leray--Hirsch theorem the cohomology ring of $F_n$ is a module over the cohomology ring of $\C P^{n+1}$, generated
by the class $y$.

The definition of Chern classes shows
\begin{equation}\label{e:co}
y^{n+1}+c_1y^{n}+\ldots + c_{n+1} = 0 \ ,
\end{equation}
where the $c_i$ are the pullbacks to the total space of the Chern classes of the base. As the total Chern class of
$\C P^{n+1}$ is given by $(1+H)^{n+2}$, the relation~\eqref{e:co} can be rewritten as $((x+y)^{n+2}-x^{n+2})/y = 0$.
\end{proof}
The Proposition holds for integral coefficients, so the class $x^{n+1}y^n$ generates the top degree integral
cohomology of $F_n$. Monomials of the form $x^my^{2n+1-m}$ vanish if $m>n+1$. The remaining relations in
top degree are given explicitly by
\begin{equation}\label{eq:exp}
x^{n+1-k}y^{n+k} = (-1)^k {n+1+k \choose k} x^{n+1}y^n \qquad \textrm{for} \qquad k\leq n+1 \ .
\end{equation}
This can easily be derived from~\eqref{eq:relFn} by induction on $k$.

The Poincar\'e polynomial of $F_n$ is
\begin{equation}\label{eq:Poincare}
P_{F_n}(t) = (1+t^2+\ldots +t^{2n})(1+t^2+\ldots +t^{2n+2}) \ .
\end{equation}
Therefore the Betti numbers of $F_n$ are
\begin{equation}\label{eq:Betti}
b_{2p}(F_n)=b_{4n+2-2p}(F_n)=p+1 \qquad \textrm{for} \qquad 0\leq p\leq n \ ,
\end{equation}
and zero otherwise. Note that additively the cohomology of $F_n$ is the same as that of $\C P^n\times\C P^{n+1}$,
but the ring structure is different.

\subsection{Failure of geometric formality}

Recall that a closed manifold is called geometrically formal if it admits a Riemannian metric for which
wedge products of harmonic forms are harmonic; cf.~\cite{K}.
We now prove the following:
\begin{thm}
For all $n\geq 1$, any closed oriented manifold $M$ with the cohomology ring of $F_n=SU(n+2)/S(U(n)\times U(1)\times U(1))$ 
is not geometrically formal.
\end{thm}
This is a consequence of the ring structure on cohomology, bringing out the difference between $F_n$ and
$\C P^n\times\C P^{n+1}$. The latter is a symmetric space, and therefore geometrically formal. The case $n=1
$ was proved in~\cite{KT}, where we also considered other homogeneous spaces $G/H$ where $H$ is a torus. 
The following proof shows that the arguments of~\cite{KT} apply much more generally.
\begin{proof}
Let $x$ and $y\in H^2(M;\Z )$ be as in Proposition~\ref{p:coho}, so that $x^{n+1}y^n$ is a generator for the top
cohomology of $M$. We can use $x$ and $z=x+y$ as a basis for the cohomology. Then $z^{n+2}=0$ by
Proposition~\ref{p:coho}, but
$$
x^n z^{n+1} = x^n (y^{n+1}+(n+1)xy^n)=-x^{n+1}y^n\neq 0
$$
by~\eqref{eq:exp}.

Suppose now that $M$ was geometrically formal. Then, identifying the harmonic forms for a formal metric with
their cohomology classes, the above relations hold for the harmonic forms. Thus $x^{n+2}=z^{n+2}=0$, but both
$x^{n+1}$ and $z^{n+1}$ are nowhere zero, because $x^{n+1}y^n = -x^n z^{n+1}$ is a volume form. Thus both $x$ and
$y$ are closed $2$-forms of rank $2n+2$, with kernels of rank $2n$.

Now rewriting~\eqref{e:co} in terms of $x$ and $z$ we obtain
$$
z^{n+1}+xz^n+x^2 z^{n-1}\ldots +x^{n+1} = 0
$$
at the level of forms. Contracting this equation with a local basis $v_1,\ldots , v_{2n}$ for the kernel of $x$,
we find
$$
i_{v_1}\ldots i_{v_{2n}} z^{n+1} + x\wedge i_{v_1}\ldots i_{v_{2n}} z^n = 0 \ .
$$
Next, contract this equation with $w$ in the kernel of $z$, to obtain
$$
i_w x\wedge i_{v_1}\ldots i_{v_{2n}} z^n = 0 \ .
$$
This implies that $x^{n+1}z^n$ cannot be a volume form, contradicting the metric formality of $M$.
\end{proof}

\subsection{K\"ahler structures and Hodge numbers}

The structure of the cohomology ring has the following implications for the Hodge and Chern numbers 
of K\"ahler structures:
\begin{thm}\label{t:BH}
    Let $M$ be any closed K\"ahler manifold with the cohomology ring of $F_n$. Then all its cohomology
    is of Hodge type $(p,p)$. In particular $h^{p,p}=b_{2p}$, and all other Hodge numbers vanish.

    The Chern numbers of $M$ satisfy
\begin{alignat}{1}
    c_{2n+1} &=(n+1)(n+2) \ , \\
    c_{1}c_{2n} &=(n+1)^3(n+2) \ .\label{eq:c1c2n}
    \end{alignat}
    \end{thm}
The statement about the Hodge structure is a generalization of the corresponding statement for homogeneous
complex structures in Lemma~\ref{l:Hodge}. Formula~\eqref{eq:c1c2n} does not hold for a non-integrable nearly
K\"ahler structure, see~\eqref{eq:nearlyc1c2n} below.
\begin{proof}
     For any K\"ahler manifold with the same cohomology ring as $F_n$ we have
     $h^{1,1}+2h^{2,0}=b_{2}(M)=b_{2}(F_n)$. As $h^{1,1}\geq 1$ and $b_{2}=2$,
     we conclude $h^{1,1}=b_{2}=2$, and $h^{2,0}=0$. By Proposition~\ref{p:coho}
     the cohomology ring is generated by $H^2(M)=H^{1,1}(M)$, and so all the
     cohomology is of type $(p,p)$.

    The top Chern number $c_{2n+1}$ is just the topological Euler number
    $P_{F_n}(-1)=(n+1)(n+2)$.

    It is known that for any compact complex manifold of complex dimension $m$ the
    Chern number $c_1c_{m-1}$ is determined by the Hodge numbers, see~\cite{LW,Sal}.
    As our $F_n$ of complex dimension $2n+1$ has the same Hodge numbers as the product
    $\C P^n\times\C P^{n+1}$, we  conclude $c_{1}c_{2n}(M)=c_{1}c_{2n}(\C P^n\times\C P^{n+1})$.
    The value of this last Chern number on $\C P^n\times\C P^{n+1}$ can be determined by a
    standard calculation.
    Alternatively, Proposition~2.3 of Libgober and Wood~\cite{LW} gives
    \begin{equation}\label{eq:LW}
    \sum_{p=2}^{2n+1}(-1)^{p}{p\choose 2}\chi_{p} =
    \frac{1}{12}\left( (2n+1)(3n-1)c_{2n+1}+c_{1}c_{2n} \right) \ ,
    \end{equation}
    with
    $$
    \chi_{p} =\sum_{q=0}^{2n+1}(-1)^q h^{p,q} \ .
    $$
    As all the cohomology is of type $(p,p)$ we obtain
    $$
    \chi_{p} =(-1)^p h^{p,p}=(-1)^pb_{2p} \ .
    $$
Substituting the values of the Betti numbers from~\eqref{eq:Betti}, and plugging the result into~\eqref{eq:LW},
a lengthy calculation involving identities for sums of binomial coefficients leads to~\eqref{eq:c1c2n}.
\end{proof}
The other Chern numbers are not in general determined by the Hodge numbers, and may vary with the complex
structure under consideration. We carry out the relevant calculations in the next two sections.

\section{Chern numbers for the standard complex structure}\label{s:standard}

Here is the general formula for the Chern classes of the standard complex structure on $F_n$:
\begin{prop}\label{tcc}
The total Chern class of $F_n$ is
\begin{equation}
c(F_n) = \frac{(1+x)^{n+2}(1+x+y)^{n+2}}{1+y} \ .
\end{equation}
\end{prop}
\begin{proof}
The fibration $p\colon F_n\rightarrow \C P^{n+1}$ is holomorphic, so we can calculate $c(F_n)$ as the product of
the total Chern classes of $p^*(T\C P^{n+1})$ and of $Tp$, the tangent bundle along the fibers. As mentioned
above, the total Chern class of $p^*(T\C P^{n+1})$ is $(1+x)^{n+2}$. For the calculation of $c(Tp)$ consider the
exact sequence
$$
L^{-1}\longrightarrow p^*(T\C P^{n+1})\longrightarrow L^{-1}\otimes Tp \ ,
$$
where $L$ is the fiberwise hyperplane bundle on the total space. As $c_1(L^{-1})=-y$, we can write formally
$$
c(Tp) = (1+x+y)^{n+2}(1-y+y^2-y^3+\ldots ) = \frac{(1+x+y)^{n+2}}{1+y}  \ .
$$
\end{proof}
Combining this with Proposition~\ref{p:coho}, one can calculate all the Chern numbers of $F_n$. The calculation is
completely elementary, but very tedious. It gives results like the following:
\begin{thm}\label{t:C}
For the standard complex structure on $F_n$, the projectivization of the tangent bundle of $\C P^{n+1}$, we have
\begin{equation}
c_1^{2n+1}(F_n) = 2 (n+1)^{n} (n+3)^{n} {2n+1 \choose n} \ ,
\end{equation}
\begin{equation}\label{eq:c1c2}
c_1^{2n-1}c_2(F_n) = 4 (n^4+7n^3+17n^2+16n+7)(n+1)^{n-2} (n+3)^{n-2} {2n-1 \choose n} \ .
\end{equation}
\end{thm}
\begin{proof}
From Proposition~\ref{tcc} we have
$$
c_1(F_n)=(n+1)(x+y)+(n+3)x \ .
$$
Using the relations $x^{n+2}=0=(x+y)^{n+2}$ and~\eqref{eq:exp}, this gives
\begin{alignat*}{1}
c_1^{2n+1}(F_n) &={2n+1 \choose n}\left( (n+1)^{n+1}(n+3)^nx^n(x+y)^{n+1}+(n+1)^n(n+3)^{n+1}x^{n+1}(x+y)^n\right) \\
&=(n+1)^n(n+3)^n{2n+1 \choose n}\left( (n+1)x^n(x+y)^{n+1}+(n+3)x^{n+1}(x+y)^n\right) \\
&=(n+1)^n(n+3)^n{2n+1 \choose n}\left( (n+1)(x^ny^{n+1}+(n+1)x^{n+1}y^n)+(n+3)x^{n+1}y^n\right) \\
&=(n+1)^n(n+3)^n{2n+1 \choose n}\cdot 2 \  .
\end{alignat*}
A similar calculation proves~\eqref{eq:c1c2} using the expression 
$$
c_2(F_n) = \frac{1}{2}(n^2+5n+8)x^2+(n^2+4n+2)x(x+y)+{n+1 \choose 2}(x+y)^2
$$ 
obtained from Proposition~\ref{tcc}.
\end{proof}
For $n=1$ Theorem~\ref{t:C} gives $c_1^3(F_1)=48$ and $c_1c_2(F_1)=24$. The latter value is actually determined
by the Hodge numbers, and can be obtained from~\eqref{eq:c1c2n} as well.

\subsection{The case $n=2$}
Here Theorem~\ref{t:BH} gives $c_5(F_2)=12$ and $c_1c_4(F_2)=108$. Theorem~\ref{t:C} gives us
$c_1^5(F_2)=4500$, which checks with the value given in~\cite{BH,H05}, and $c_1^3c_2(F_2)=2148$. In this case
it remains to calculate $c_1c_2^2$, $c_1^2c_3$ and $c_2c_3$. We do this in some detail
in order to illustrate some shortcuts in calculations making Proposition~\ref{tcc} explicit. These shortcuts are
useful when calculating for $F_n$ with larger $n$.

The tangent bundle of $F_n$ has a complex splitting into $p^*(T\C P^{n+1})$ and $Tp$, which have almost equal ranks.
Therefore, computing certain Chern classes of $F_n$ using the Whitney sum formula there are not too many summands.
By the proof of Proposition~\ref{tcc}, the total Chern class of $Tp$ is
$$
c(Tp) = \frac{(1+x+y)^{n+2}}{1+y} \ .
$$
However, as the rank of $Tp$ is $n$, we can truncate this at terms of degree $n$. In the case at hand $n=2$, and we
have
$$
c(Tp) = 1+ (4x+3y) + (6x^2+8xy+3y^2) \ .
$$
Combining this with $c(p^*(T\C P^{n+1}))=1+4x+6x^2+4x^3$ and using the Whitney sum formula, we find
\begin{alignat*}{1}
c_1(F_2) &= 8x+3y \ , \\
c_2(F_2) &=28x^2+20xy+3y^2 \ , \\
c_3(F_2) &=52x^3+50x^2y+12xy^2 \ .
\end{alignat*}
Multiplying out using $x^4=0$, and substituting from~\eqref{eq:exp}, we quickly obtain the numbers given in the 
first column of the table in the introduction.

\subsection{The case $n=3$}
Here Theorem~\ref{t:BH} gives $c_7(F_3)=20$ and $c_1c_6(F_3)=320$. Theorem~\ref{t:C} gives us
$c_1^7(F_3)=967680$ and $c_1^5c_2(F_3)=458880$. We have completed the calculation of all the Chern
numbers in this case using the procedure outlined above. The results are presented in the table in Figure~\ref{table3}.
We shall not reproduce the details of the calculation here, but we mention some of the intermediary steps.

Here $TF_3=p^*(T\C P^4)\oplus Tp$ splits as a direct sum of complex vector bundles of rank $4$ and $3$
respectively. To find the Chern classes of $Tp$ we look at
$$
c(Tp) = \frac{(1+x+y)^{5}}{1+y}
$$
and ignore all terms of degree larger than $3$ to obtain
$$
c(Tp) = 1+ (5x+4y)+(10x^2+15xy+6y^2)+(10x^3+20x^2y+15xy^2+4y^3) \ .
$$
Multiplying $(1+x)^5$ with $c(Tp)$ and using $x^5=0$ together with~\eqref{eq:exp}, we find, in addition to
$c_1(F_3)=10x+4y$ and $c_2(F_3)= 45x^2+35xy+6y^2$, which were already mentioned in the proof of 
Theorem~\ref{t:C}, the following:
\begin{alignat*}{1}
c_3(F_3) &= 120x^3+135x^2y+45xy^2+4y^3 \ , \\
c_4(F_3) &= 5(41x^4+58x^3y+27x^2y^2+4xy^3) \ , \\
c_5(F_3) &= 10(37x^4y+21x^3y^2+4x^2y^3) \ .
\end{alignat*}
Multiplying out and again using~\eqref{eq:exp}, we obtain the numbers given in the table in Figure~\ref{table3}.

\section{Chern numbers of the twistor space}\label{ss:Cherntwistor}

Let us denote by $Z_n$ the twistor space of the Grassmannian $G_n$. Then $Z_n$ is diffeomorphic to $F_n$, 
the projectivization of $T\C P^{n+1}$, but has a different complex structure, as described in Section~\ref{s:cx}. 
The complex structure of the twistor space is in fact given by the projectivization of the cotangent bundle of 
$\C P^{n+1}$, and the holomorphic contact structure of the twistor space mentioned in Section~\ref{s:cx} is the 
tautological contact structure of $\PP (T^*\C P^{n+1})$; see~\cite{Boothby2,Wolf,Bianchi,LeS}.

To calculate the Chern numbers of $Z_n=\PP (T^*\C P^{n+1})$ we shall follow the same approach as 
for the projectivization of the tangent bundle. First we write down the cohomology ring in a way which is adapted to 
the projectivization of the cotangent bundle:
\begin{prop}\label{p:dualcoho}
The cohomology ring of $Z_n$ is generated by two elements $x$ and $z$ of degree $2$, subject to the relations
\begin{equation}\label{eq:relZn}
x^{n+2}=0 \ , \qquad \ \frac{(x-z)^{n+2}-x^{n+2}}{z} = 0 \ .
\end{equation}
\end{prop}
We omit  the proof because it is exactly the same as that of Proposition~\ref{p:coho}.
The Proposition holds for integral coefficients, so the class $x^{n+1}z^n$ generates the top degree integral
cohomology of $Z_n$. Monomials of the form $x^mz^{2n+1-m}$ vanish if $m>n+1$. The remaining relations in
top degree are given explicitly by
\begin{equation}\label{eq:dualexp}
x^{n+1-k}z^{n+k} = {n+1+k \choose k} x^{n+1}z^n \qquad \textrm{for} \qquad k\leq n+1 \ .
\end{equation}
This can easily be derived from~\eqref{eq:relZn} by induction on $k$.

Next we determine the total Chern class of $Z_n$.
\begin{prop}\label{dualtcc}
In the generators $x$ and $z$ from Proposition~\ref{p:dualcoho}, the total Chern class of $Z_n$ is
\begin{equation}
c(Z_n) = \frac{(1+x)^{n+2}(1-x+z)^{n+2}}{1+z} \ .
\end{equation}
\end{prop}
Again we omit the proof, because it is exactly the same as the proof of Proposition~\ref{tcc}.

\begin{rem}
The complex anti-linear isomorphism between $T\C P^{n+1}$ and $T^*\C P^{n+1}$ induces a diffeomorphism between
$F_n$ and $Z_n$ which pulls back $x$ to $x$ and $z$ to $-y$. In the basis $x$ and $y$ which we used for $F_n$,
the total Chern class of $Z_n$ is:
$$
c(Z_n) = \frac{(1+x)^{n+2}(1-x-y)^{n+2}}{1-y} \ .
$$
Note that the relations in the cohomology ring are neater when  expressed in terms of $x$ and $z$, rather than in terms of 
$x$ and $y$. On the top degree generators, $x^{n+1}z^n$ is pulled back to $(-1)^nx^{n+1}y^n$. Thus, the diffeomorphism is 
orientation-preserving if and only if $n$ is even. For $n$ odd we get different generators in top degree, and we may have to 
replace one of the complex structures by its conjugate to get the same orientation. 
\end{rem}

Combining Propositions~\ref{p:dualcoho} and~\ref{dualtcc}, one can calculate all the Chern numbers of $Z_n$. 
The calculation is again completely elementary, but rather tedious, although it is a little less so than for the standard
complex structure, due to the more convenient presentation of the cohomology ring, and an easier to handle formula
for the first Chern class. This calculation leads to results like the following:
\begin{thm}\label{t:dualC}
For the projectivization $Z_n$ of the cotangent bundle of $\C P^{n+1}$, we have
\begin{equation}\label{eq:dualc1}
c_1^{2n+1}(Z_n) = (n+1)^{2n+1}  {2n+2 \choose n+1} \ ,
\end{equation}
\begin{equation}\label{eq:dualc1c2}
c_1^{2n-1}c_2(Z_n) = 2 (n+1)^{2n-1} (n^2+n+1) {2n \choose n} \ ,
\end{equation}
\begin{equation}\label{eq:dualc1c22}
c_1^{2n-3}c_2^2(Z_n) = (n+1)^{2n-3} n (4n^3+8n^2+10n+5){2n-2 \choose n-1} \ .
\end{equation}
\end{thm}
\begin{proof}
The previous Proposition gives in particular $c_1(Z_n)=(n+1)z$ and 
$$
c_2(Z_n) = -(n+2)x^2+(n+2)xz+ {n+1 \choose 2}z^2 \ .
$$
From this one computes mechanically using the relations~\eqref{eq:dualexp}.
\end{proof}

For $n=1$ we find $c_1^3(Z_1)=48$ and $c_1c_2(Z_1)=24$. These are of course the same values as for $F_1$, 
compare Subsection~\ref{ss:complete}. 
However, for larger $n$ we find in particular:
\begin{cor}
For all $n>1$ one has $c_1^{2n+1}(Z_n)\neq c_1^{2n+1}(F_n)$. 
\end{cor}

\begin{ex}
For $n=2$ Theorem~\ref{t:dualC} gives $c_1c_2^2(Z_2)=1068$, $c_1^3c_2(Z_2)=2268$ and $c_1^5(Z_2) = 4860$. This last 
value checks with a value given in~\cite{BH,H05}. From Theorem~\ref{t:BH} we have $c_1c_4(Z_2)=108$, so that the only Chern
numbers left to compute for $Z_2$ are $c_1^2c_3$ and $c_2c_3$. Using the procedure applied to $F_2$ we determine an 
explicit formula for $c_3(Z_2)$, and multiplying out gives the numbers in the middle column of the table in Figure~\ref{table2}. 
\end{ex}

\begin{ex}
For $n=3$ Theorems~\ref{t:BH} and~\ref{t:dualC} tell us some of the Chern numbers. To calculate all of them we can apply
the method  outlined in the previous section. For the Chern classes we know already that $c_1(Z_3)=4z$ and 
$c_2(Z_3)=-5x^2+5xz+6z^2$, and now we find
\begin{alignat*}{1}
c_3(Z_3) &= -15x^3+15xz^2+4z^3 \ , \\
c_4(Z_3) &= 5(x^4-2x^3z-3x^2z^2+4xz^3) \ , \\
c_5(Z_3) &= 10(7x^4z-9x^3z^2+4x^2z^3) \ .
\end{alignat*}
This leads to the Chern numbers given in the middle column of the table in Figure~\ref{table3}.
\end{ex}

To end this section we discuss the relationship between our calculations and a special case of those of Semmelmann and Weingart~\cite{SW}.
The holomorphic line bundle $L$ on the twistor space is ample, because $L^{n+1}=K^{-1}$ and $K^{-1}$ is ample for 
any complex manifold with a K\"ahler--Einstein metric of positive scalar curvature. Thus one can consider $(Z_n,L)$ as a 
polarised projective algebraic variety with Hilbert polynomial
$$
P(r)=\chi (Z_n,\OO (L^r))=\sum_{i=0}^{2n+1}(-1)^i\dim_{\C}H^i(Z_n,\OO (L^r)) \ .
$$
By the Hirzebruch--Riemann--Roch theorem, this can be calculated as
$$
P(r) = \langle ch(L^r)Todd(Z_n),[Z_n]\rangle \ ,
$$
which is a polynomial of degree $2n+1$ in $r$. As we know all the Chern classes of $Z_n$ and the Chern class of $L$,
we can in principle calculate the Hilbert polynomial. Conversely, if we know the Hilbert polynomial, then we can read off all the
combinations of Chern numbers which appear in it as coefficients of powers of $r$. Let us just write out the terms
of highest degree in $r$:
\begin{alignat*}{1}
P(r) = &\frac{1}{(2n+1)! \ (n+1)^{2n+1}}c_1(Z_n)^{2n+1}r^{2n+1}\\
+ &\frac{1}{2 \ (2n)! \ (n+1)^{2n}}c_1(Z_n)^{2n+1}r^{2n}\\
+ &\frac{1}{12 \ (2n-1)! \ (n+1)^{2n-1}}(c_1(Z_n)^{2n+1}+c_1(Z_n)^{2n-1}c_2(Z_n))r^{2n-1}+...
\end{alignat*}
Now Semmelmann and Weingart~\cite{SW} have calculated the Hilbert polynomial of the twistor space of the
Grassmannian explicitly:
$$
P(r) = \frac{n+2r+1}{n+1}{n+r \choose r}^2 \ .
$$
Expanding this in powers of $r$ we find:
$$
 \frac{n+2r+1}{n+1}{n+r \choose r}^2= \frac{2}{n! \ (n+1)!}r^{2n+1}
+ \frac{2n+1}{(n!)^2}r^{2n} + \frac{3n^2+4n+2}{3 \ (n-1)! \ n!}r^{2n-1}+...
$$
Comparing the coefficients of $r^{2n+1}$ in the two expansions, we find:
\begin{equation}\label{eq:SW1}
c_1^{2n+1}(Z_n) = 2(n+1)^{2n+1}{2n+1 \choose n} \ .
\end{equation}
This agrees with~\eqref{eq:dualc1}.

One can determine further combinations of Chern numbers for $Z_n$ by looking at the terms of lower order in $r$. The
coefficients of $r^{2n}$ give no new information, but provide a consistency check for the calculation of
$c_1^{2n+1}(Z_n)$. Combining this calculation with the comparison of the coefficients of $r^{2n-1}$, we find
\begin{equation}\label{eq:SW2}
c_1^{2n-1}c_2(Z_n) = 4(n^2+n+1)(n+1)^{2n-1}{2n-1 \choose n} \ .
\end{equation}
This agrees with~\eqref{eq:dualc1c2}.

One could calculate some more Chern numbers by looking at the further terms in the expansions, but this would not 
be enough to compute all the Chern numbers of $Z_n$.

\section{Chern numbers of the nearly K\"ahler structure}\label{ss:Chernnearly}

We denote by $N_n$ the smooth manifold underlying $F_n$ and $Z_n$, but endowed with the non-integrable almost
complex structure $\hat J$ that is part of the nearly K\"ahler structure defined at the end of Section~\ref{s:cx}.
The Chern classes of $N_n$ are given by the following:
\begin{prop}\label{nearlytcc}
In the generators $x$ and $z$ from Proposition~\ref{p:dualcoho}, the total Chern class of $N_n$ is
\begin{equation}
c(N_n) = c(Z_n)\cdot \frac{1-z}{1+z} = \frac{(1+x)^{n+2}(1-x+z)^{n+2}(1-z)}{(1+z)^2} \ .
\end{equation}
\end{prop}
\begin{proof}
The second equality follows from Proposition~\ref{dualtcc}. To prove the first equality, recall from Section~\ref{s:cx}
that as complex vector bundles we have $TZ_n=L\oplus D$ and $TN_n=L^{-1}\oplus D$. Thus, for the total Chern
classes we find
$$
c(Z_n)=(1+c_1(L))\cdot c(D) \quad \quad \textrm{and}  \quad \quad c(N_n)=(1-c_1(L))\cdot c(D) \ .
$$
Furthermore, we have $(n+1)z=c_1(Z_n)=(n+1)c_1(L)$. As the cohomology of $Z_n$ is torsion-free, we conclude
$c_1(L)=z$, which completes the proof.
\end{proof}
Combining Propositions~\ref{p:dualcoho} and~\ref{nearlytcc}, one can calculate all the Chern numbers of $N_n$. 
This gives results like the following:
\begin{thm}\label{t:nearlyC}
For the nearly K\"ahler manifold $N_n$ we have
\begin{equation}\label{eq:nearlyc1}
c_1^{2n+1}(N_n) = -(n-1)^{2n+1}  {2n+2 \choose n+1} \ ,
\end{equation}
\begin{equation}\label{eq:nearlyc1c2}
c_1^{2n-1}c_2(N_n) = -4 (n-1)^{2n-1} (n^2-n-1) {2n-1 \choose n+1} \ ,
\end{equation}
\begin{equation}\label{eq:nearlyc1c22}
c_1^{2n-3}c_2^2(N_n) = -\frac{1}{n+1} \cdot (n-1)^{2n-3} (4n^5-20n^4+34n^3-17n^2-11n+16){2n-2 \choose n-1} \ ,
\end{equation}
\begin{equation}\label{eq:nearlyc1c2n}
c_1c_{2n}(N_n)=(n-1)^2(n+1)(n+2) \ .
\end{equation}
\end{thm}
\begin{proof}
The previous Proposition gives in particular $c_1(N_n)=(n-1)z$ and 
$$
c_2(N_n) = -(n+2)x^2+(n+2)xz+ \frac{1}{2}n(n-3)z^2 \ .
$$
From this one computes~\eqref{eq:nearlyc1}, \eqref{eq:nearlyc1c2} and~\eqref{eq:nearlyc1c22} mechanically using the 
relations~\eqref{eq:dualexp}. The only new feature is that the almost 
complex structure $\hat J$ of $N_n$ induces the orientation opposite to the one induced by the complex structure of the 
twistor space $Z_n$. Therefore $x^{n+1}z^n$ is now the negative rather than the positive generator of the top degree 
cohomology. 

One can also use the argument from the proof of Proposition~\ref{nearlytcc} to calculate Chern 
numbers of $N_n$ from those of the twistor space $Z_n$. We use this approach to prove~\eqref{eq:nearlyc1c2n}.

Recall from Section~\ref{s:cx} that as complex vector bundles $TZ_n=L\oplus D$ and $TN_n=L^{-1}\oplus D$, 
and that $c_1(Z_n)=(n+1)c_1(L)$ and $c_1(N_n)=(n-1)c_1(L)$. For the Chern classes $c_1c_{2n}$ this means 
$$
c_1c_{2n}(N_n)= (n-1)(c_1(L)c_{2n}(D)-c_1^2(L)c_{2n-1}(D))
$$
and
$$
c_1c_{2n}(Z_n)= (n+1)(c_1(L)c_{2n}(D)+c_1^2(L)c_{2n-1}(D)) \ .
$$
Evaluating on the fundamental class of $Z_n$, the second equation gives the following relation between Chern numbers:
$$
(n+1)^3(n+2) = (n+1)((n+1)(n+2)+\langle c_1^2(L)c_{2n-1}(D),[Z_n]\rangle ) \ ,
$$
where we have used~\eqref{eq:c1c2n} on the left hand side, and we have used the known value for the 
top Chern number $c_{2n+1}(Z_n)$ to identify the first term on the right hand side. Now we can similarly
evaluate the first equation on $[N_n]=-[Z_n]$ and plug in what we just computed for the evaluation of 
$c_1^2(L)c_{2n-1}(D)$ to obtain~\eqref{eq:nearlyc1c2n}.
\end{proof}

\begin{ex}
For $n=2$ Theorem~\ref{t:nearlyC} gives $c_1^5(N_2)=-20$, $c_1^3c_2(N_2)=c_1c_2^2(N_2)=-4$ and $c_1c_4(N_2)=12$.
In this case we can easily extract $c_3(N_2)$ from the formula in Proposition~\ref{nearlytcc} and 
carry out the multiplication in the cohomology ring to prove $c_1^2c_3(N_2)=20$ and $c_2c_3(N_2)=4$.
\end{ex}

\begin{ex}
The case $n=3$ is also fairly easy because the formula for $c_2$ simplifies to $c_2(N_3)=-5(x^2-xz)$. This immediately 
yields $c_1c_2^3(N_3)=-500$, in addition to  the values already given by the theorem. Still, to calculate all the Chern 
numbers more work is needed. From Proposition~\ref{nearlytcc}, together  with our calculation of the Chern classes 
of $Z_3$, we find the following:
\begin{alignat*}{1}
c_3(N_3) &= -5x^2z+5xz^2-2z^3 \ , \\
c_4(N_3) &= 5x^4-10x^3z+5x^2z^2-2z^4 \ , \\
c_5(N_3) &= 2(30x^4z-35x^3z^2+25x^2z^3-10xz^4+z^5) \ .
\end{alignat*}
This leads to the numbers given in the third column of the table in Figure~\ref{table3}.
\end{ex}

\section{Final remarks}\label{s:final}

In this section we explain the relationship between the different points of view on the almost Hermitian structures that 
we have discussed. 

First of all, the holomorphic tangent and cotangent bundles of $\C P^{n+1}$ are homogeneous bundles under 
$SU(n+2)$, and therefore the complex structures of their projectivizations, denoted $F_n$ and $Z_n$ in 
Sections~\ref{s:cx} to~\ref{ss:Cherntwistor}, are also homogeneous under $SU(n+2)$. Thus, up to conjugation, 
they must equal the invariant complex structures $I$ and $J$ in Proposition~\ref{roots}, but {\it a priori} it is not 
clear which is which, and the case $n=1$ shows that distinguishing between the two is not an entirely trivial matter. 
By looking at the Chern classes we can however immediately say that the standard complex structure $F_n$ is $I$ 
and the twistor space structure $Z_n$ is $J$. This follows most easily by looking at the divisibilities of $c_1$. On the 
projectivized cotangent bundle the divisibility is a multiple of $n+1$ due to the presence of a holomorphic contact 
structure. This fits with the formula for $c_1(J)$ in Example~\ref{ex:c1}, but not with $c_1(I)$.

The fibration of $F_n$ over the Grassmannian $G_n$ is a homogeneous fibration, and the tangent bundle along
the fibers is given by the two-dimensional irreducible subrepresentation $R_0$ of the isotropy representation of 
$F_n$, compare Subsection~\ref{ABH}. In the definition of the nearly K\"ahler structure $\hat J$ in Section~\ref{s:cx}
we started with the complex structure of the twistor space and conjugated it along the fiber of the twistor fibration.
This matches precisely the relationship between the $J$ and $\hat J$ in Proposition~\ref{roots},
which coincide on $R_1$ and $R_2$ but are conjugate to each other on $R_0$. Thus the $\hat J$ of Section~\ref{s:cx}
is the same as the homogeneous $\hat J$ of Proposition~\ref{roots} in Section~\ref{s:Lie}.

As the fibration of $F_n$ over the Grassmannian $G_n$ is homogeneous, with the tangent bundle along the fibers
corresponding to a subrepresentation of  the isotropy representation, we can modify any homogeneous metric on 
the total space of the fibration by constant rescaling along the fibers leaving the orthogonal complement unchanged,
and the resulting metric will still be homogeneous. This just means that on the summand $R_0$ of the 
isotropy representation we change the metric by multiplication with a constant. Therefore the nearly K\"ahler metric 
$\hat g$ defined in Section~\ref{s:cx} is homogeneous. 

This scaling procedure can be applied to any Riemannian submersion with totally geodesic fibers, and is sometimes 
called the canonical variation of the submersion metric, see~\cite{Besse}, 9G. It is a standard way to build new
Einstein metrics from old ones. For the twistor fibration of $Z_n$ over $G_n$ one has a K\"ahler--Einstein metric 
on $Z_n$, and its canonical variation contains another Einstein but non-K\"ahler metric, see~\cite{AGI,Besse}. This 
Einstein metric is also homogeneous, and coincides with the nearly K\"ahler metric if and only if $n=1$, as one sees
by comparing Theorem~3.1 and Proposition~3.2 of~\cite{AGI}.

\section*{Appendix: Chern numbers for $n=3$}

The table in Figure~\ref{table3} summarizes our calculations of the Chern numbers for the case $n=3$. Up to
complex conjugation, the three columns correspond to the almost complex structures $I$, $J$ and $\hat J$ from 
Section~\ref{s:Lie}. These were denoted by $F_n$ (standard structure), $Z_n$ and $N_n$ in later sections.

\bigskip

\begin{figure}
{
\begin{tabular}{|l|c|c|c|} \hline
& & & \\ 
& \raisebox{1.5ex}[-1.5ex]{standard structure $\PP (T\C P^4)$} & \raisebox{1.5ex}[-1.5ex]{twistor space $\PP (T^{*}\C P^4)$} & \raisebox{1.5ex}[-1.5ex]{nearly K\"ahler structure} \\ \hline\hline
{\rule [-3mm]{0mm}{8mm}$c_1^7$ } 
& $967680$ & $1146880$ & $-8960$  \\ \hline
{\rule [-3mm]{0mm}{8mm}$c_1^5c_2$ } 
& $458880$ & $532480$ & $-3200$  \\ \hline
{\rule [-3mm]{0mm}{8mm}$c_1^3c_2^2$ } 
& $217680$ & $247680$ & $-1200$ \\ \hline
{\rule [-3mm]{0mm}{8mm}$c_1c_2^3$ } 
& $103330$ & $115480$ & $-500$ \\ \hline
{\rule [-3mm]{0mm}{8mm}$c_1^4c_3$ } 
& $134080$ & $148480$ & $640$ \\ \hline
{\rule [-3mm]{0mm}{8mm}$c_1^2c_2c_3$ } 
& $63580$ & $69280$ & $200$ \\ \hline
{\rule [-3mm]{0mm}{8mm}$c_1c_3^2$ } 
& $18530$ & $19480$ & $-60$ \\ \hline
{\rule [-3mm]{0mm}{8mm}$c_2^2c_3$ } 
& $30180$ & $32430$ & $50$ \\ \hline
{\rule [-3mm]{0mm}{8mm}$c_1^3c_4$ } 
& $26320$ & $27520$ & $880$ \\ \hline
{\rule [-3mm]{0mm}{8mm}$c_1c_2c_4$ } 
& $12470$ & $12920$ & $300$ \\ \hline
{\rule [-3mm]{0mm}{8mm}$c_3c_4$ } 
& $3620$ & $3670$ & $-70$ \\ \hline
{\rule [-3mm]{0mm}{8mm}$c_1^2c_5$ } 
& $3520$ & $3520$ & $400$ \\ \hline
{\rule [-3mm]{0mm}{8mm}$c_2c_5$ } 
& $1670$ & $1670$ & $150$ \\ \hline
{\rule [-3mm]{0mm}{8mm}$c_1c_6$ } 
& $320$ &  $320$ & $80$ \\ \hline
{\rule [-3mm]{0mm}{8mm}$c_7$ } 
& $20$ & $20$ & $20$ \\ \hline
\end{tabular}
}
\caption{The Chern numbers for the invariant almost Hermitian structures of $F_3$}\label{table3}
\end{figure}

\bigskip

\newpage

\bibliographystyle{amsplain}

\bigskip

\end{document}